\documentclass{amsart}   	
\usepackage{amssymb,amsthm}
\usepackage{amsthm}
\usepackage{graphicx}
\usepackage{enumitem}
\usepackage{tikz}
\usepackage{tikz-cd}
\usetikzlibrary{calc}
\usetikzlibrary{decorations.pathmorphing}
\usepackage{hyperref}
\usepackage{adjustbox}
\usepackage{mathtools}
\usepackage[all]{xy}

\newtheorem{theorem}{Theorem}[section]
\newtheorem{proposition}[theorem]{Proposition}
\newtheorem{lemma}[theorem]{Lemma}
\newtheorem{corollary}[theorem]{Corollary}

\theoremstyle{definition}
\newtheorem{definition}[theorem]{Definition}
\newtheorem{construction}[theorem]{Construction}

\newtheorem{example}[theorem]{Example}

\newcommand{\qedno}{\nobreak\hfill\ensuremath{\square}}

\begin{document}
\title{Graphs with equal girth and circumference}
\author{Lewis Stanton*, Jeffrey Thompson}

\begin{abstract} 
We characterise the form of all simple, finite graphs for which the girth of the graph is equal to the circumference of the graph. We apply this to prove a bound on the number of edges in such a graph.
\end{abstract}

\maketitle 

\section{Introduction}
\label{sec:intro} 

Cycles are a natural object which arise from the definition of a graph, and the invariants related to cycles have been extensively studied. In particular, the girth and circumference of a graph (being the length of the shortest and longest cycle in a graph respectively) are invariants which arise naturally from the definition of a cycle. An interesting problem is to try to determine the structure of graphs for which the girth and circumference have specific values. In this paper, we characterise the graphs for which the girth and circumference are equal. In other words, graphs for which all of the cycles contained within it have the same length.

Similar questions to this have been studied before. For example, \cite{AH} considered graphs $G$ for which $G$ and its complement have the same girth. Graphs $G$ such that $G$ and its complement have circumference 3 or 4 were also characterised. The opposite extreme to having a graph $G$ with equal cycle lengths has also been studied. Erd\"os posed the problem of determining the maximum number $f(n)$ of edges in a simple graph with $n$ vertices in which any two cycles are of different lengths \cite{BM}[p. 247, Problem 11]. There are currently lower bounds \cite{S} and upper bounds \cite{BCFY} on $f(n)$, however this is still an open problem.

In this paper, we fully characterise graphs $G$ for which the girth and circumference are equal. First it is shown that $G$ must be planar. Now let $G$ be a finite, simple and planar graph containing at least one cycle. Two faces $F$ and $F'$ of $G$ are adjacent if the cycles bounding them share an edge, and $F$ and $F'$ are connected if there is a sequence of faces $F_1,\cdots,F_m$ where $F_1 = F$, $F_m = F'$ and $F_{k}$ is adjacent to $F_{k+1}$ for all $1 \leq k \leq m-1$. A \textit{strict face component} $G_F$ containing a face $F$ of $G$ consists of the vertices and edges of the cycles bounding the faces connected to $F$. The characterisation of graphs with equal girth and circumference is in terms of its strict face components, as in Corollary ~\ref{generalcharacterisation}. This leads to an alternative characterisation in Theorem ~\ref{blockgeneralcharacterisation} in terms of blocks, which can be used to give an algorithm for determining if a given graph has equal girth and circumference. Finally in Section ~\ref{sec:bound}, we give an upper bound on the number of edges in a finite, simple, connected and planar graph $G$ on $n$ vertices with equal girth and circumference $r$. In particular, if the value of $r$ is known, it is proved in Theorem ~\ref{upperboundwithr} that if $r$ is even, then $|E(G)| \leq n-1+\left\lfloor\frac{n-\frac{r}{2}-1}{\frac{r}{2}-1}\right\rfloor$, or if $r$ is odd, then $|E(G)| \leq n-1+\left\lfloor\frac{n-1}{r-1}\right\rfloor$. If the value of $r$ is unknown, it is proved in Corollary ~\ref{upperboundwithoutr} that $|E(G)|\leq 2n-4$. This has applications in showing whether a given graph must have two cycles of different lengths, solely in terms of the number of vertices and edges.

\section{Basic Definitions}
\label{sec:definitions}

In this section, we define the basic notions in graph theory which will be required in this paper. We will use the definitions from Gross, Yellen and Zhang \cite{GYZ}. We restrict our attention to simple, finite and connected graphs. There are two invariants of graphs related to cycles known as the girth and the circumference. These invariants are well defined when the graph $G$ contains a cycle. Connected graphs which do not contain cycles are known as \textit{trees}. The \textit{girth} of $G$, denoted girth$(G)$, is the length of the shortest cycle in $G$. The \textit{circumference} of $G$, denoted circum$(G)$, is the length of the longest cycle in $G$.

Let $H$ be a subgraph of a graph $G$. The values of the girth and circumference of $H$ and $G$ can be related through an inequality. A graph invariant $\mu$ is \textit{monotone} if $\mu(H) \leq \mu(G)$ for all simple, connected and finite subgraphs $H$ of a simple, connected and finite graph $G$. A graph invariant $\mu$ is \textit{anti-monotone} if $\mu(G) \leq \mu(H)$ for all simple, connected and finite subgraphs $H$ of a simple, connected and finite graph $G$. It follows from the definitions of girth and circumference that the girth of a graph is anti-monotone and the circumference of a graph is monotone.

It will be shown that a necessary condition for a graph with girth$(G) = \text{circum}(G)$ is for a graph to be planar. A graph $G$ is \textit{planar} if it can be embedded in $\mathbb{R}^2$. An embedding of a planar graph $G$ into the plane is known as a \textit{plane representation} of $G$. A \textit{face} of a given plane representation of a graph $G$ is a component of the complement of the plane representation in the plane. For a given plane embedding of a graph $G$, we will need to differentiate between two types of faces. An \textit{internal face} of a graph $G$ is a bounded component of the complement of the plane representation in the plane. The \textit{external face} of a graph $G$ is the unbounded component of the complement of the plane representation in the plane. The structure of the internal faces of $G$ will turn out to be closely related to subgraphs known as blocks. A \textit{cut vertex} $v$ of $G$ is a vertex $v$ such that removing it disconnects $G$. A \textit{block} in a graph $G$ is a maximally connected subgraph that does not contain a cut vertex. If $G$ is connected and contains no cut vertices, then $G$ is called a block graph.

We will also require the notion of homeomorphic graphs. This involves an operation on graphs known as subdivision of an edge. For any edge $e = x \cdot y \in E(G)$ in a graph $G$, we can form a new graph $G'$ by introducing a new vertex $z$ in the interior of $e$. This forms two new edges in $G'$, namely $e_1 = x \cdot z$ and $e_2 = z \cdot y$. A graph $G'$ is a \textit{subdivision} of a graph $G$ if we can obtain $G'$ by subdividing the edges of $G$. 

Before defining a homeomorphism, we define the notion of a homomorphism between graphs. Let $G$ and $H$ be graphs. A \textit{homomorphism} between $G$ and $H$ is a function $\phi:V(G) \rightarrow V(H)$ such that if $v$ is adjacent to $w$ in $G$, $\phi(v)$ is adjacent to $\phi(w)$ in $H$. A bijective homomorphism is called an \textit{isomorphism}. If there exists an isomorphism between two graphs $G$ and $H$, we say that $G$ and $H$ are \textit{isomorphic} and we denote this by $G \cong H$. Two graphs are homeomorphic if there exist subdivisions $G'$ and $H'$ of $G$ and $H$ respectively such that $G' \cong H'$.

Finally, we define special types of graphs which will be used throughout the paper. The \textit{complete graph} on $n$ vertices, denoted $K_n$, is the graph such that every pair of distinct vertices are adjacent. There is a variant of the complete graph for bipartite graphs. For $m \geq 1$ and $n \geq 1$, the \textit{complete bipartite graph}, denoted $K_{m,n}$, is the graph whose vertex set is the union of two sets $V(K_{m,n}) = A \cup B$, where $A$ and $B$ are disjoint sets with $|A| = m$ and $|B| = n$. The edge set $E(K_{n,m})$ consists of all edges of the form $e = a \cdot b$ where $a \in A$ and $b \in B$. The \textit{path graph} $P_n$ is a tree on $n+1$ vertices with 2 vertices of degree 1 and $n-1$ vertices of degree 2. Finally, the \textit{cycle graph} on $n$ vertices, denoted $C_n$ is the graph consisting of a single cycle of length $n$.

\section{Characterising graphs with equal girth and circumference}
\label{sec:char}

\subsection{Non-Planar Graphs}
\label{subsec:Non-plane}

In this section, it is proven that non-planar graphs $G$ have girth$(G) \neq \text{circum}(G)$. From this point onwards, any references to faces will be taken to be internal faces unless otherwise specified. To do this, a well known characterisation of planar graphs will be required known as Kuratowski's theorem \cite{K}.

\begin{theorem}
\label{kuratowski}
A simple, connected and finite graph $G$ is non-planar if and only if $G$ contains a subgraph that is homeomorphic to either $K_5$ or to $K_{3,3}$.
\qedno
\end{theorem} We will also require the following lemma.

\begin{lemma}
\label{subgraphequalG&C}
Let $G$ be a simple, connected and finite graph with subgraph $H$ which contains a cycle. If girth$(G) = \text{circum}(G)$, then girth$(G) = \text{girth}(H) = \text{circum}(H)$.
\end{lemma}
\begin{proof}
Let $\mathcal{C}$ be the set of cycles in $G$ and $\mathcal{C}' \subseteq \mathcal{C}$ be the set of cycles in $H$. Note that by assumption, $\mathcal{C}'$ is non-empty. Each cycle $C \in \mathcal{C}$ must all be of the same length $l = \text{girth}(G)$ since girth$(G) = \text{circum}(G)$. Therefore, all the cycles in $H$ must have length $l = \text{girth}(G)$ and so girth$(G) = \text{girth}(H) = \text{circum}(H)$.
\end{proof}

Using Theorem ~\ref{kuratowski} and Lemma ~\ref{subgraphequalG&C}, we can prove the desired result. 

\begin{theorem}
\label{nonplanarnotequal}
Let $G$ be a simple, connected and finite non-planar graph. Then girth$(G) \neq \text{circum}(G)$.
\end{theorem}
\begin{proof}
It suffices to find a subgraph $H$ with girth$(H) \neq \text{circum}(H)$. This is because if $H$ is a subgraph with girth$(H) \neq \text{circum}(H)$, then $H$ contains two cycles which have two distinct lengths. These cycles will also be contained in $G$ and so girth$(G) \neq \text{circum}(G)$. By Theorem ~\ref{kuratowski}, there exists a subgraph $H$ in $G$ which is homeomorphic to either $K_5$ or to $K_{3,3}$. Therefore, it suffices to prove that any subgraph $H$ homeomorphic to $K_5$ or to $K_{3,3}$ has girth$(H) \neq \text{circum}(H)$.

First, consider the case where $H$ is homeomorphic to $K_5$ and suppose $H$ satisfies girth$(H) = \text{circum}(H)$. This means $H$ can be obtained by subdividing the edges of $K_5$. Let $e_1,\cdots,e_{10}$ be the edges of $K_5$ and denote by $d_i$ the number of times the edge $e_i$ is subdivided to obtain $H$. Let $C_1,\cdots,C_{10}$ be the 3-cycles in $K_5$. 

To ensure that girth$(H) = \text{circum}(H)$, the length of each 3-cycle after subdividing must equal a fixed number $N \geq 3$. In particular, this means that for each 3-cycle $C_k$ where $k \in \{1,\cdots,10\}$, $\sum_{e_i \in C_k} d_i = N$. Each edge $e_i$ of $K_5$ is contained in at least one cycle of $K_5$ and so this gives a system of 10 equations with 11 unknowns. As the reader can easily check, the dimension of the null space of the matrix associated with this system of equations is 1. Since each cycle $C_k$ contains three edges, letting $d_i = d$ for any $d \in \mathbb{N}_0$ is the only solution to this system of equations. Hence for girth$(H)$ to be equal to circum$(H)$, each edge must be subdivided $d$ times. However, $K_5$ contains a 5-cycle. Subdividing the edge of a 5-cycle $d$ times gives a cycle of length $5+5d$, and so circum$(H) = 5+5d$. Subdividing the edge of a 3-cycle $d$ times gives a cycle of length $3+3d$, and so girth$(H) = 3+3d$. These are not equal for $d \geq 0$. Hence, girth$(H) \neq \text{circum}(H)$. 

Now, consider the case where $H$ is homeomorphic to $K_{3,3}$ and suppose $H$ satisfies girth$(H) = \text{circum}(H)$. This means $H$ can be obtained by subdividing the edges of $K_{3,3}$. Let $e_1,\cdots,e_{9}$ be the edges of $K_{3,3}$ and denote by $d_i$ the number of times the edge $e_i$ must be subdivided to obtain $H$. Let $C_1,\cdots,C_{9}$ be the 4-cycles in $K_5$. 

To ensure that girth$(H) = \text{circum}(H)$, the length of each 4-cycle after subdividing must equal a fixed number $N \geq 4$. In particular, this means that for each 4-cycle $C_k$ where $k \in \{1,\cdots,9\}$, $\sum_{e_i \in C_k} d_i = N$. Each edge $e_i$ of $K_{3,3}$ is contained in at least one cycle of $K_{3,3}$ and so this gives a system of 9 equations with 10 unknowns. As the reader can easily check, the dimension of the null space of the matrix associated with this system of equations is 1. Since each cycle $C_k$ contains four edges, letting $d_i = d$ for any $d \in \mathbb{N}_0$ is the only solution to this system of equations. Therefore for girth$(G)$ to be equal to circum$(H)$, each edge must be subdivided $d$ times. However, $K_{3,3}$ contains a 6-cycle. Subdividing the edge of a 6-cycle $d$ times gives a cycle of length $6+6d$, and so circum$(H) = 6+6d$. Subdividing the edge of a 4-cycle $d$ times gives a cycle of length $4+4d$, and so girth$(H) = 4+4d$. These are not equal for $d \geq 0$. Hence girth$(H) \neq \text{circum}(H)$. 

\end{proof}

\subsection{Planar Graphs}
\label{subsec:Plane}

Theorem ~\ref{nonplanarnotequal} implies that any graph with girth$(G) = \text{circum}(G)$ must be planar. To characterise planar graphs with girth$(G) = \text{circum}(G)$, we first define subgraphs of $G$ which we will consider as building blocks for $G$. This involves introducing the notion of a strict face-connected graph.

\begin{definition}
Suppose $G$ is a finite, simple and planar graph. Let $F$ and $F'$ be two faces of $G$ and let $C$ and $C'$ be the cycles which bound them. The faces $F$ and $F'$ are \textit{adjacent} if $C$ and $C'$ share at least one common edge.
\end{definition}

\begin{definition}
Let $G$ be a finite, simple, planar graph containing at least one cycle. Two faces $F$ and $F'$ in $G$ are connected if there is a sequence of faces $F_1,\cdots,F_m$ where $F_1 = F$, $F_m = F'$ and $F_{k}$ is adjacent to $F_{k+1}$ for all $1 \leq k \leq m-1$. The graph $G$ is \textit{face-connected} if  any two faces $F_i$ and $F_j$ of $G$ are connected. 
\end{definition}

\begin{definition}
A finite, simple and planar graph $G$ containing a cycle is \textit{strict face-connected} if it is face-connected and the vertices and edges of $G$ are the union of the vertices and edges of the cycles bounding the internal faces.
\end{definition}

The following figure gives an example of a graph that is strict face-connected ($G_1$), a graph that is face-connected but not strict face-connected ($G_2$) and a graph that is not face-connected ($G_3$).

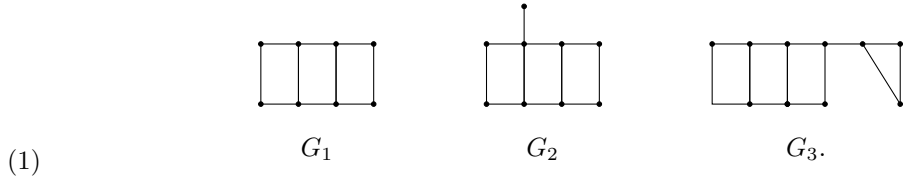
\begin{equation}\label{eq:faceconnectedex}{\begin{tikzpicture}
     \draw (-3.0,-0.1) -- (-2.5,-0.1) -- (-2.5,-0.9) -- (-3,-0.9) -- (-3.0,-0.1); 
		 \draw (-2.5,-0.1) -- (-2,-0.1) -- (-2,-0.9) -- (-2.5,-0.9) -- (-2.5,-0.1);
		 \draw (-2,-0.1) -- (-1.5,-0.1) -- (-1.5,-0.9) -- (-2,-0.9) -- (-2,-0.1);
      
     \draw [fill] (-3.0,-0.1) circle [radius=0.03];
		 \draw [fill] (-2.5,-0.1) circle [radius=0.03];
		 \draw [fill] (-2.5,-0.9) circle [radius=0.03];
		 \draw [fill] (-2,-0.9) circle [radius=0.03];
		 \draw [fill] (-2,-0.1) circle [radius=0.03];
		 \draw [fill] (-3,-0.9) circle [radius=0.03];
		 \draw [fill] (-1.5,-0.1) circle [radius=0.03];
		 \draw [fill] (-1.5,-0.9) circle [radius=0.03];
		 \node at (-2.25,-1.5) {$G_1$};

     \draw (0.0,-0.1) -- (0.5,-0.1) -- (0.5,-0.9) -- (0.0,-0.9) -- (0.0,-0.1); 
		 \draw (0.5,-0.1) -- (1.0,-0.1) -- (1.0,-0.9) -- (0.5,-0.9) -- (0.5,-0.1);
		 \draw (1.0,-0.1) -- (1.5,-0.1) -- (1.5,-0.9) -- (1.0,-0.9) -- (1.0,-0.1);
		 \draw (0.5,-0.1) -- (0.5,0.4);
      
     \draw [fill] (0.0,-0.1) circle [radius=0.03];
		 \draw [fill] (0.5,-0.1) circle [radius=0.03];
		 \draw [fill] (0.5,-0.9) circle [radius=0.03];
		 \draw [fill] (0.0,-0.9) circle [radius=0.03];
		 \draw [fill] (1.0,-0.1) circle [radius=0.03];
		 \draw [fill] (1.0,-0.9) circle [radius=0.03];
		 \draw [fill] (1.5,-0.1) circle [radius=0.03];
		 \draw [fill] (1.5,-0.9) circle [radius=0.03];
		 \draw [fill] (0.5,0.4) circle [radius=0.03];
		 \node at (0.75,-1.5) {$G_2$};
		
		     \draw (3,-0.1) -- (3.5,-0.1) -- (3.5,-0.9) -- (3,-0.9) -- (3,-0.1); 
		 \draw (3.5,-0.1) -- (4,-0.1) -- (4,-0.9) -- (3.5,-0.9) -- (3.5,-0.1);
		 \draw (4,-0.1) -- (4.5,-0.1) -- (4.5,-0.9) -- (4,-0.9) -- (4,-0.1);
		 \draw (4.5,-0.1) -- (5,-0.1);
		 \draw (5,-0.1) -- (5.5,-0.1) -- (5.5,-0.9) -- (5.5,-0.9) -- (5,-0.1);
      
     \draw [fill] (3,-0.1) circle [radius=0.03];
		 \draw [fill] (3.5,-0.1) circle [radius=0.03];
		 \draw [fill] (4,-0.9) circle [radius=0.03];
		 \draw [fill] (3.5,-0.9) circle [radius=0.03];
		 \draw [fill] (4,-0.1) circle [radius=0.03];
		 \draw [fill] (4,-0.9) circle [radius=0.03];
		 \draw [fill] (4.5,-0.1) circle [radius=0.03];
		 \draw [fill] (4.5,-0.9) circle [radius=0.03];
		 \draw [fill] (5,-0.1) circle [radius=0.03];
		 \draw [fill] (5.5,-0.9) circle [radius=0.03];
		 \draw [fill] (5.5,-0.1) circle [radius=0.03];
		 \node at (4.25,-1.5) {$G_3$.};
		
  \end{tikzpicture}}\end{equation}

\begin{definition}
Let $G$ be a finite, simple and planar graph containing at least one cycle. Let $F$ be a face of $G$. A \textit{strict face component} $G_F$ containing $F$ consists of the vertices and edges of the cycles bounding the faces connected to $F$.
\end{definition}

In (\ref{eq:faceconnectedex}), since $G_1$ and $G_2$ are face-connected, they have one strict face component although note that the strict face component of $G_2$ is not the whole graph. The graph $G_3$ has two strict face components, namely the sequence of 4-cycles on the left and the one 3-cycle on the right.

The notion of a strict face component is closely related to blocks. To see this, we first prove how strict face components are related. 

\begin{lemma}
\label{facecomponentintersection}
Let $G$ be a simple, finite and planar graph and let $\mathcal{F}_1,\cdots,\mathcal{F}_k$ be the strict face components of $G$. Then for any two distinct strict face components $\mathcal{F}_i$ and $\mathcal{F}_j$, $E(\mathcal{F}_i) \cap E(\mathcal{F}_j) = \emptyset$. Moreover, $V(\mathcal{F}_i) \cap V(\mathcal{F}_j) = \{v\}$, where $v$ is a cut vertex of $G$ or $V(\mathcal{F}_i) \cap V(\mathcal{F}_j) = \emptyset$.
\end{lemma}

\begin{proof}
Let $\mathcal{F}_i$ and $\mathcal{F}_j$ be two distinct strict face components and suppose $E(\mathcal{F}_i) \cap E(\mathcal{F}_j) \neq \emptyset$. This implies there exist faces $F_i$ in $\mathcal{F}_i$ and $F_j$ in $\mathcal{F}_j$ which are adjacent, and so $\mathcal{F}_i = \mathcal{F}_j$.

Now suppose $V(\mathcal{F}_i) \cap V(\mathcal{F}_j) = \{v_1,\cdots,v_l\}$ for $l \geq 2$. Observe that $\mathcal{F}_i$ is connected as if not, for a face $F$ in a connected component $C$, any vertex not contained in $C$ can not be in the strict face component containing $F$. Therefore, there exists a path $P_i$ contained in $\mathcal{F}_i$ between $v_1$ and $v_2$. Similarly, there exists a path $P_j$ in $\mathcal{F}_j$ between $v_1$ and $v_2$. Since the edge sets of $\mathcal{F}_i$ and $\mathcal{F}_j$ are disjoint, the edge sets of these paths must be disjoint. This implies that $v_1$ and $v_2$ are part of a cycle $C$ which has edges contained in both $\mathcal{F}_i$ and $\mathcal{F}_j$. Therefore, $C$ has an edge adjacent to a face in $\mathcal{F}_i$ and an edge adjacent to a face in $\mathcal{F}_j$, and so it follows that $\mathcal{F}_i = \mathcal{F}_j$. Hence, either $V(\mathcal{F}_i) \cap V(\mathcal{F}_j) = \{v\}$ or $V(\mathcal{F}_i) \cap V(\mathcal{F}_j) = \emptyset$.

Finally, suppose $V(\mathcal{F}_i) \cap V(\mathcal{F}_j) = \{v\}$ and $v$ is not a cut vertex. Since $\mathcal{F}_i$ and $\mathcal{F}_j$ are connected, there exist vertices $v_i \in \mathcal{F}_i$ and $v_j \in \mathcal{F}_j$ which are adjacent to $v$. Since $v$ is not a cut vertex, removing $v$ from $G$ results in a graph $G'$ which is still connected. Therefore there exists a path $P$ between $v_i$ and $v_j$ contained in $G'$. Adding $v$ back in gives a cycle $C$ containing $v$ in $G$ which contains the edges $v \cdot v_i \in E(\mathcal{F}_i)$ and $v \cdot v_j \in E(\mathcal{F}_j)$. Hence, $C$ is a cycle containing an edge in both $\mathcal{F}_i$ and $\mathcal{F}_j$, which is a contradiction.
\end{proof}

In the following, we will require the notion of the wedge sum of graphs.

\begin{definition}
Let $G$ and $H$ be simple graphs and let $v_G \in V(G)$ and $v_H \in V(H)$ be chosen vertices, known as \textit{base vertices}. The wedge sum $G \vee H$ of $G$ and $H$ is the graph formed from the disjoint union of $G$ and $H$ by identifying the base vertices. Define the identified vertex in $G \vee H$ as the base vertex of the wedge sum.
\end{definition}

Observe that the base vertex of $G \vee H$ is a cut vertex.

\begin{proposition}
\label{strictfaceblock}
A planar graph $G$ containing a cycle is strict face-connected iff $G$ is a planar block graph.
\end{proposition}
\begin{proof}
Suppose $G$ is a planar block graph containing a cycle but not strict face-connected. By definition of block, this implies that $G$ contains no cut vertices and so also has no bridges. Therefore, the only possibility for $G$ is to contain at least 2 strict-face components. Since $G$ is connected and contains no bridges, by Lemma ~\ref{facecomponentintersection}, there must exist two strict face-components $\mathcal{F}_1$ and $\mathcal{F}_2$ such that $V(\mathcal{F}_i) \cap V(\mathcal{F}_j) = \{v\}$ where $v$ is a cut vertex, which is a contradiction.

Now suppose $G$ is strict face-connected but not a block graph. This implies that $G$ must contain a cut vertex $v$. Removing $v$ gives a graph $G'$ with two connected components $C_1$ and $C_2$. Therefore, every path in $G$ between vertices $v_1 \in C_1$ and $v_2 \in C_2$ must pass through $v$. It follows that $G = H_1 \vee H_2$ for some graphs $H_1$ and $H_2$. However, no face $F_1$ in $H_1$ can be adjacent to a face $F_2$ in $H_2$ since $E(H_1) \cap E(H_2) = \emptyset$. Therefore, $G$ is not strict face-connected which is a contradiction.
\end{proof}

Proposition ~\ref{strictfaceblock} shows that each strict face component of a planar graph $G$ is a block. However, the converse is not true. For example for $G_2$ in (\ref{eq:faceconnectedex}), the induced subgraph on the vertex of degree one along with its adjacent vertex forms a block, however this is not part of a strict face component. 

Let $G$ be a finite, simple and planar graph with strict face components $\mathcal{F}_1,\cdots,\mathcal{F}_m$. Observe that each cycle of $G$ is contained within a single strict face component. This observation gives the following result.

\begin{proposition}
\label{girthoffacecomponents}
Let $G$ be a finite, simple and planar graph with strict face components denoted $\mathcal{F}_1,\cdots,\mathcal{F}_m$. Then \[\text{girth}(G) = \min \{\text{girth}(\mathcal{F}_1),\cdots,\text{girth}(\mathcal{F}_m)\} \text{ and } \text{circum}(G) = \min \{\text{circum}(\mathcal{F}_1),\cdots,\text{circum}(\mathcal{F}_m)\}.\]
\qedno
\end{proposition}

Proposition ~\ref{girthoffacecomponents} implies that if $G$ is a finite, simple and planar graph with girth$(G) = \text{circum}(G)$, then each face component must also have equal girth and circumference. Therefore, it suffices to consider strict face-connected graphs. First, we will introduce a collection of graphs known as generalised book graphs. A \textit{generalised book graph} is denoted $B(n,L,p)$ where $1 \leq n \leq L-2$, $L \geq 3$ and $p \geq 2$. These graphs will be the building blocks for graphs with equal girth and circumference. The graph $B(n,L,p)$ is $p$ cycles $C_L$ of length $L$ glued together over a common path $P_n$ of length $n$. For example, $B(1,3,2)$ is the graph 
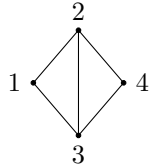
\[\begin{tikzpicture} 
	   \draw (-0.1,-0.1) -- (0.5,0.6) -- (0.5,-0.8) -- (-0.1,-0.1);
     \draw (0.5,0.6) -- (1.1,-0.1) -- (0.5,-0.8);
		 \draw [fill] (-0.1,-0.1) circle [radius=0.03];
		 \draw [fill] (0.5,0.6) circle [radius=0.03];
		 \draw [fill] (0.5,-0.8) circle [radius=0.03];
		 \draw [fill] (1.1,-0.1) circle [radius=0.03];
	   \node at (-0.35,-0.1) {1};
		 \node at (0.5,0.85) {2};
		 \node at (0.5,-1.05) {3};
		 \node at (1.35,-0.1) {4};
  \end{tikzpicture}\] and the generalised book graph $B(1,3,3)$ is \[\scalebox{1.5}
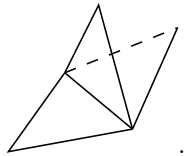
{\begin{tikzpicture} 
     \draw (-0.1,-0.1) -- (0.5,-0.6); 
		 \draw (-0.1,-0.1) -- (0.2,0.5) -- (0.5,-0.6); 
		 \draw (0.9,0.3) -- (0.5,-0.6); 
		 \draw [dashed] (-0.1,-0.1) -- (0.9,0.3);
		 \draw (-0.1,-0.1) -- (-0.6,-0.8) -- (0.5,-0.6); 
  \end{tikzpicture}}.\] For the theorem that follows, we draw the generalised book graph in its planar form. For example $B(2,4,4)$ is a graph of the form  \[\scalebox{1.5}
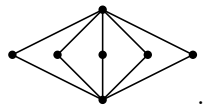
{\begin{tikzpicture}
     \draw (-0.1,-0.1) -- (-0.1,-0.9); 
		 \draw (-0.1,-0.1) -- (-0.5,-0.5) -- (-0.1,-0.9); 
		 \draw (-0.1,-0.1) -- (0.3,-0.5) -- (-0.1,-0.9); 
		 \draw (-0.1,-0.1) -- (0.7,-0.5) -- (-0.1,-0.9);
		 \draw (-0.1,-0.1) -- (-0.9,-0.5) -- (-0.1,-0.9);
      
     \draw [fill] (-0.1,-0.1) circle [radius=0.03];
		 \draw [fill] (-0.1,-0.5) circle [radius=0.03];
		 \draw [fill] (-0.1,-0.9) circle [radius=0.03];
		 \draw [fill] (-0.5,-0.5) circle [radius=0.03];
		 \draw [fill] (0.3,-0.5) circle [radius=0.03];
		 \draw [fill] (-0.9,-0.5) circle [radius=0.03];
		 \draw [fill] (0.7,-0.5) circle [radius=0.03];
  \end{tikzpicture}}.\] Observe that the generalised book graph $B(n,L,p)$ has $p$ faces and is strict face-connected. We can now prove the characterisation theorem for planar graphs with girth equal to circumference.

\begin{theorem}
\label{characterisation}
Let $G$ be a finite, simple and strict face-connected graph with $\text{girth}(G) = \text{circum}(G)$. Then if $G$ has $p = 1$ face, then $G = C_l$ for some $l \geq 3$. If $G$ has $p \geq 2$ faces, then $\text{girth}(G) = \text{circum}(G) = 2k$ for some $k \in \mathbb{N}$, $k\geq 2$ and $G = B(k,2k,p)$.
\end{theorem}
\begin{proof}
First, if $G=C_m$ for some $m$ then the result is clear. Therefore we consider the case where $G$ has two or more faces.

We proceed by induction. Let $G$ be a simple, strict face-connected and finite graph with two faces, $F_1$ and $F_2$ with cycles $C_1$ bounding $F_1$ and $C_2$ bounding $F_2$. Since $\text{girth}(G) = \text{circum}(G)$, $C_1$ and $C_2$ have length $L$ for some $L \geq 3$. The faces $F_1$ and $F_2$ are adjacent, which implies that the cycles $C_1$ and $C_2$ must intersect along a common path $P_k$ for some $k \geq 1$. This creates a new cycle $C_3$ which starts at the endpoint of $P_k$, goes around $C_1$ (in the direction which does not travel along $P_k$) to the other endpoint of $P_k$ and then around $C_2$ (in the direction which does not travel along $P_k$). This is illustrated by the following diagram. 
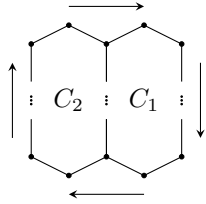
\[\begin{tikzpicture}

     \draw [-stealth](-0.5,1.25) -- (0.5,1.25);
		 \draw [-stealth](1.25,0.5) -- (1.25,-0.5);
		 \draw [-stealth](0.5,-1.25) -- (-0.5,-1.25);
		 \draw [-stealth](-1.25,-0.5) -- (-1.25,0.5);
		
     \draw (0,0.75) -- (0,0.25);
		 \draw (0,-0.75) -- (0,-0.25);
		 \draw [fill] (0,0.75) circle [radius=0.03];
		 \draw [fill] (0,-0.75) circle [radius=0.03];
		 \draw(0,0.25)--(0,0.75);
		 \draw(0,-0.25)--(0,-0.75);
		 \draw [fill] (0,0.05) circle [radius=0.01];
		 \draw [fill] (0,0) circle [radius=0.01];
		 \draw [fill] (0,-0.05) circle [radius=0.01]; 
	
     \draw (0,0.75) -- (-0.5,1);
		 \draw (-1,0.75) -- (-0.5,1);
		 \draw (-1,0.75) -- (-1,0.25);
		
		 \draw [fill] (-0.5,1) circle [radius=0.03];
		 \draw [fill] (-1,0.75) circle [radius=0.03];
		
     \draw (0,0.75) -- (0.5,1);
		 \draw (1,0.75) -- (0.5,1);
		 \draw (1,0.75) -- (1,0.25);
		
		 \draw [fill] (0.5,1) circle [radius=0.03];
		 \draw [fill] (1,0.75) circle [radius=0.03];
		
		 \draw (0,-0.75) -- (-0.5,-1);
		 \draw (-1,-0.75) -- (-0.5,-1);
		 \draw (-1,-0.75) -- (-1,-0.25);
		
		 \draw [fill] (-0.5,-1) circle [radius=0.03];
		 \draw [fill] (-1,-0.75) circle [radius=0.03];
		
     \draw (0,-0.75) -- (0.5,-1);
		 \draw (1,-0.75) -- (0.5,-1);	
		 \draw (1,-0.75) -- (1,-0.25);
		
		 \draw [fill] (0.5,-1) circle [radius=0.03];
		 \draw [fill] (1,-0.75) circle [radius=0.03];
		
		 \draw [fill] (1,0.05) circle [radius=0.01];
		 \draw [fill] (1,0) circle [radius=0.01];
		 \draw [fill] (1,-0.05) circle [radius=0.01];
		
		 \draw [fill] (-1,0.05) circle [radius=0.01];
		 \draw [fill] (-1,0) circle [radius=0.01];
		 \draw [fill] (-1,-0.05) circle [radius=0.01];
		
		 \node at (0.5,0) {$C_1$};
		
		 \node at (-0.5,0) {$C_2$};

  \end{tikzpicture}\] This cycle has length $2(L-k)$. Since $\text{girth}(G) = \text{circum}(G)$, $2(L-k) = L$ which implies that $L=2k$. Therefore, the cycle must have length $2k$ and so $G = B(k,2k,2)$.

Now consider a simple, strict face-connected and finite graph $G$ with three faces $F_1$, $F_2$ and $F_3$ bounded by cycles $C_1$, $C_2$ and $C_3$. By assumption, $\text{girth}(G) = \text{circum}(G)$ and so these cycles have a common length $L$. Since $G$ is face-connected, $C_1$ must be adjacent to another face. Suppose without loss of generality $C_1$ is adjacent to $C_2$. Considering the induced subgraph $H$ on the vertices of $C_1$ and $C_2$, the $n=2$ case implies $H = B(k,2k,2)$ for some $k \geq 2$. By Lemma ~\ref{subgraphequalG&C}, the lengths of the cycles in $H$ must be equal to the lengths of the cycles in $G$ and so $L=2k$. By definition of $B(k,2k,2)$, $C_1$ and $C_2$ intersect over a common path of length $k$.

Now consider the cycle $C_3$ and assume without loss of generality that $C_3$ is adjacent to $C_2$. By a similar argument, considering the induced subgraph $H'$ on the vertices of $C_2$ and $C_3$, $C_2$ and $C_3$ must intersect over a common path of length $k$. Therefore, $C_1$ and $C_3$ are both adjacent to $C_2$ over a path of length $k$. However, the edge set of these paths of length $k$ must be disjoint as $G$ is planar and so $G = B(k,2k,3)$.

Suppose the result is true for a strict face-connected graph with $k \leq m$ faces and consider a simple, strict face-connected and finite graph $G$ with $m+1$ faces, $m \geq 3$. Let $F_1,\cdots,F_{m+1}$ be the faces of $G$ and let $C_i$ be the cycle bounding the face $F_i$ for $1 \leq i \leq m$. Since $\text{girth}(G) = \text{circum}(G)$, $C_i$ must have a common length $L$ for all $i$. Suppose without loss of generality that $F_{m+1}$ is adjacent to the external face of $G$. Consider the induced subgraph $H$ on the vertices of the remaining cycles $C_1,\cdots,C_m$. Observe that $H$ is a graph with $m$ faces and so by the inductive hypothesis has the form $B(k,2k,m)$ for some $k \geq 2$. By Lemma ~\ref{subgraphequalG&C}, the lengths of the cycles in $H$ must be equal to the lengths of the cycles in $G$ and so $L=2k$. Also by the structure of the generalised book graph, we can order the cycles $C_1, \cdots, C_m$ such that consecutive cycles intersect along a common path of length $k$ and any two non-consecutive cycles $C_i$ and $C_j$ intersect at two vertices $\{v_0,v_1\}$ independent of $i$ and $j$. 

Now consider the remaining face $F_{m+1}$. Since $G$ is face-connected, there exists $1 \leq l \leq m$ such that $F_{n+1}$ is adjacent to $F_{l}$. Suppose $1 < l < m$. The induced subgraph $H_l$ on the vertices of $C_l$ and $C_{m+1}$ has two faces and so $H_l = B(k_l,2k_l,2)$. By Lemma ~\ref{subgraphequalG&C}, the lengths of the cycles in $H_l$ must be equal to the lengths of the cycles in $G$ and so $k_l = k$. This implies that $C_{m+1}$ intersects $C_{l}$ over a path of length $k$. However by definition of $H$, this implies that $F_{m+1}$ is also adjacent to $F_{l-1}$ or $F_{l+1}$. Suppose without loss of generality, that $F_{m+1}$ is adjacent to $F_{l+1}$ and consider the induced subgraph $H_{l+1}$ on the vertices of $F_{m+1}$ and $F_{l+1}$. By the same argument, $H_{l+1} = B(k,2k,2)$ and $C_{m+1}$ intersects $C_{l+1}$ over a path of length $k$. The cycle $C_{m+1}$ is adjacent to both $C_{l}$ and $C_{l+1}$ over paths of length $k$. However since $G$ is planar, the edge set of these paths must be disjoint which cannot happen as $F_l$ and $F_{l+1}$ are adjacent. Therefore, $C_{m+1}$ cannot be adjacent to $C_l$ for $1 < l < m$.

Now suppose without loss of generality that $F_{m+1}$ is adjacent to $F_m$ and is not adjacent to $F_{l}$ for $1<l<m$. Consider the induced subgraph $H$ on the vertices of $C_m$ and $C_{m+1}$. By a similar argument to the previous case, the cycles $C_{m}$ and $C_{m+1}$ must intersect over a common path of length $k$. Now, suppose $F_{m+1}$ is also adjacent to $F_1$. By considering the induced subgraph $H'$ on the vertices of $C_m$ and $C_1$, the cycles $C_{m}$ and $C_{1}$ must intersect over a common path of length $k$. However since $G$ is planar, the edge set of these paths must be disjoint which cannot happen as $F_{m+1}$ is an internal face. Therefore, $C_{m+1}$ cannot be adjacent to $C_1$. Hence $G = B(k,2k,m+1)$ as desired.
\end{proof}

Combining Proposition ~\ref{girthoffacecomponents} and Theorem ~\ref{characterisation} gives the following corollary.

\begin{corollary}
\label{generalcharacterisation}
Let $G$ be a finite, simple and planar graph with $r = \text{girth}(G) = \text{circum}(G)$. Let $\mathcal{F}_1,\cdots,\mathcal{F}_m$ be the strict face components of $G$. If $r$ is even, then $\mathcal{F}_i = B\left(\frac{r}{2},r,p\right)$ or $\mathcal{F}_i = C_r$. If $r$ is odd, then $\mathcal{F}_i = C_r$.
\qedno
\end{corollary}

This fully characterises the structure of a graph with girth equal to circumference. However given a simple, planar graph $G$, it may not be easy to determine if $G$ has the required form. Proposition ~\ref{strictfaceblock} can be used to rephrase Corollary ~\ref{generalcharacterisation} in terms of blocks in order to determine if $G$ has the required form algorithmically. In particular, let $G$ be a finite, simple and planar graph and define $G_B$ to be the graph obtained from removing the bridges in $G$, and then removing the isolated vertices. Since an edge is a bridge iff it is not contained in a cycle, $G_B$ is the union of the strict face components of $G$. By Proposition ~\ref{strictfaceblock}, these strict face components are blocks and we obtain the following result.

\begin{theorem}
\label{blockgeneralcharacterisation}
Let $G$ be a finite, simple and planar graph and $H_1,\cdots, H_k$ be the blocks of $G_B$. Then $G$ has $r = \text{girth}(G) = \text{circum}(G)$ iff for all $1 \leq i \leq k$, $H_i = B\left(\frac{r}{2},r,p_i\right)$ or $H_i = C_r$ and $r$ is even, or $H_i = C_r$ for all $1 \leq i \leq k$, if $r$ is odd.
\qedno
\end{theorem}

\section{Upper bounds on the number of edges in a graph with equal girth and circumference}
\label{sec:bound}

In this section, we prove upper bounds on the number of edges in a graph with girth equal to circumference. To do this, we will consider a general construction of planar graphs. Let $G_1,\cdots, G_k$ be graphs. Construct a graph $G'_2$ from the disjoint union of $G_1$ and $G_2$ by identifying a vertex in $G_1$ and a vertex in $G_2$. Iteratively define $G'_i$ for $3 \leq i \leq k$ from the disjoint union of $G'_{i-1}$ and $G_i$ by identifying a vertex in $G'_{i-1}$ and a vertex in $G_i$. Since at each stage of the construction, we are identifying a single vertex of two graphs together, we obtain the following result.

\begin{lemma}
\label{invariantvertexchoice}
The number of vertices and edges in the graphs $G_1', \cdots, G_k'$ are invariant under the choice of vertex identification.
\qedno
\end{lemma} 

\begin{construction}
\label{planargraphconstruction}
Let $G$ be a simple, finite and planar graph and let $\mathcal{F}_1,\cdots,\mathcal{F}_k$ be the strict face components of $G$. Lemma ~\ref{facecomponentintersection} implies that any edge contained in a cycle in $G$ must be contained in a unique strict face component of $G$. Conversely, any edge not contained in any cycle is a bridge. Therefore, we can view $G$ as strict face components joined together by a single cut vertex (by Lemma ~\ref{facecomponentintersection}) or trees. This implies that we can construct any planar graph $G$ iteratively in the following way. Define a graph $G_2$ as the disjoint union of a strict face component and another strict face component or a tree as appropriate and identifying a single vertex in each graph. Continue this process by defining $G_i$ to be the disjoint union of $G_{i-1}$ and a strict face component or a tree and identifying a single vertex in each graph. Since $G$ is a finite graph, this process will terminate giving $G = G_n$ for some $n \geq 1$. \end{construction}

In what follows, we will only be concerned with the number of edges and vertices contained within $G$. Lemma ~\ref{invariantvertexchoice} implies that the choice of vertex identification does not affect the number of vertices or edges in the graph. Therefore, it will be simplest to consider the wedge sum of the strict face components and trees. 

The number of vertices and edges of a wedge sum can be determined easily in terms of the summands. At each stage, only one vertex is being identified. Therefore, the number of edges is unaffected, and two vertices in the disjoint union become one vertex in the wedge sum. Therefore, we obtain the following.

\begin{lemma}
\label{wedgevertandedge}
Let $G_1,\cdots,G_n$ be graphs, then \[\left\lvert V\left(\bigvee\limits_{i=1}^n G_i\right) \right\rvert = \sum\limits_{i=1}^n |V(G_i)|-(n-1), \text{ } \left\lvert E\left(\bigvee\limits_{i=1}^n G_i\right) \right\rvert = \sum\limits_{i=1}^n |E(G_i)|.\]
\qedno
\end{lemma}

By Corollary ~\ref{generalcharacterisation}, Construction ~\ref{planargraphconstruction} and Lemma ~\ref{invariantvertexchoice}, to consider the number of edges in a finite, simple and planar graph $G$ with $r= \text{girth}(G) = \text{circum}(G)$, it suffices to consider $\bigvee_{i=1}^{k_1} \mathcal{F}_i \vee \bigvee_{j=1}^{k_2} T_j$ where $T_j$ is a tree and $\mathcal{F}_i = B(\frac{r}{2},r,p_i)$ for some $p_i \geq 2$ or $\mathcal{F}_i = C_r$ if $r$ is even, or $\mathcal{F}_i = C_r$ if $r$ is odd. Further, without loss of generality we can consider $G' = \bigvee_{i=1}^m \mathcal{F}_i \vee P_k$ by Lemma ~\ref{invariantvertexchoice}. In the case that $r$ is even, we can view $C_r$ as two paths of length $\frac{r}{2}$ glued together by their endpoints, therefore it makes sense to define $B(\frac{r}{2},r,1)$ as $C_r$ in this case. 

We now study the properties of the wedge sum of generalised book graphs and path graphs which will then be applied to determine a bound on the number of edges in $G$. Viewing $B(k,2k,p)$ as $p+1$ path graphs $P_k$ of length $k$ glued together by their end points, we obtain the following result.

\begin{lemma}
\label{vertandedgebook}
Let $B(k,2k,p)$ be a generalised book graph with $k\geq 1$ and $p \geq 1$. Then \[\left\lvert V(B(k,2k,p))\right\rvert = (k-1)(p+1)+2, \text{ } \left\lvert E(B(k,2k,p))\right\rvert = k(p+1).\]
\qedno
\end{lemma}

Now, for a fixed number of vertices, we bound the number of edges of the wedge sum of two generalised book graphs.

\begin{lemma}
\label{twobookgraph}
Let $B(\frac{r}{2},r,a)$ and $B(\frac{r}{2},r,b)$ be generalised book graphs with $r \geq 4$, $r$ even and $a,b \geq 1$. Let \[n = \left\lvert V\left(B\left(\frac{r}{2},r,a\right) \vee B\left(\frac{r}{2},r,b\right)\right)\right\rvert \text{, } p = \left\lfloor \frac{\left(\frac{r}{2}-1\right)(a+b+1)+1}{\frac{r}{2}-1}\right\rfloor,\text{ } c = n-2-\left(\frac{r}{2}-1\right)\left(p+1\right).\] Then \[(i) \left\lvert V\left(B\left(\frac{r}{2},r,a\right) \vee B\left(\frac{r}{2},r,b\right)\right) \right\rvert = \left\lvert V\left(B\left(\frac{r}{2},r,p\right) \vee P_{c}\right)\right\rvert\]\[(ii) \left\lvert E\left(B\left(\frac{r}{2},r,a\right) \vee B\left(\frac{r}{2},r,b\right)\right) \right\rvert \leq \left\lvert E\left(B\left(\frac{r}{2},r,p\right) \vee P_{c}\right)\right\rvert.\]
\end{lemma}

\begin{proof}
First, by Lemma ~\ref{wedgevertandedge} and Lemma ~\ref{vertandedgebook}, \[n = \left(\frac{r}{2}-1\right)(a+1)+2+\left(\frac{r}{2}-1\right)(b+1)+2-1 = 2+\left(\frac{\left(\frac{r}{2}-1\right)(a+b+1)+1}{\frac{r}{2}-1}+1\right)\left(\frac{r}{2}-1\right)\]\[\geq 2+\left(p+1\right)\left(\frac{r}{2}-1\right) = \left\lvert V(B(\frac{r}{2},r,p))\right\rvert.\] Therefore $c \geq 0$, and so $P_c$ is well defined. Moreover, $(i)$ follows from Lemma ~\ref{wedgevertandedge} and Lemma ~\ref{vertandedgebook}.

For $(ii)$, consider $\left\lvert E\left(B\left(\frac{r}{2},r,a\right) \vee B\left(\frac{r}{2},r,b\right)\right)\right\rvert$. By Lemma ~\ref{wedgevertandedge} and Lemma ~\ref{vertandedgebook}, \[\left\lvert E\left(B\left(\frac{r}{2},r,a\right) \vee B\left(\frac{r}{2},r,b\right)\right)\right\rvert = \frac{r}{2}(a+1)+\frac{r}{2}(b+1)\]\[ = \left\lfloor \frac{\left(\frac{r}{2}-1\right)(a+b+1)}{\frac{r}{2}-1}\right\rfloor + \left(\frac{r}{2}-1\right)(a+b+2)+1 \leq \left\lfloor \frac{\left(\frac{r}{2}-1\right)(a+b+1)+1}{\frac{r}{2}-1}\right\rfloor + \left(\frac{r}{2}-1\right)(a+b+2)+1\]\[ < \left\lfloor \frac{\left(\frac{r}{2}-1\right)(a+b+1)+1}{\frac{r}{2}-1}\right\rfloor + \left(\frac{r}{2}-1\right)(a+b+2) +2 = \left\lfloor \frac{\left(\frac{r}{2}-1\right)(a+b+1)+1}{\frac{r}{2}-1}\right\rfloor +n-1\]\[ = \left\lvert E\left(B\left(\frac{r}{2},r,p\right) \vee P_c\right)\right\rvert.\]
\end{proof}

Next, for a fixed number of vertices, we bound the number of edges of the wedge sum of a generalised book graph and a path graph.

\begin{lemma}
\label{booktree}
Let $B(\frac{r}{2},r,p)$ and $B(\frac{r}{2},r,p')$ be generalised book graphs with $r \geq 4$ where $r$ is even, and $1 \leq p' \leq p$. Let $P_c$ and $P_{c'}$ be path graphs with $0 \leq c \leq c'$. If $\left\lvert V\left(B\left(\frac{r}{2},r,p\right) \vee P_c\right) \right\rvert = \left\lvert V\left(B\left(\frac{r}{2},r,p'\right) \vee P_{c'}\right) \right\rvert$, then \[\left\lvert E\left(B\left(\frac{r}{2},r,p'\right) \vee P_{c'}\right) \right\rvert \leq \left\lvert E\left(B\left(\frac{r}{2},r,p\right)\vee P_c\right)\right\rvert.\]
\end{lemma}
\begin{proof}
By Lemma ~\ref{wedgevertandedge} and Lemma ~\ref{vertandedgebook}, the assumption $\left\lvert V\left(B\left(\frac{r}{2},r,p\right) \vee P_c\right) \right\rvert = \left\lvert V\left(B\left(\frac{r}{2},r,p'\right) \vee P_{c'}\right) \right\rvert$ implies that $c' = \left(\frac{r}{2}-1\right)(p-p')+c$. Now by Lemma ~\ref{wedgevertandedge} and Lemma ~\ref{vertandedgebook}, \[\left\lvert E\left(B\left(\frac{r}{2},r,p'\right) \vee P_{c'}\right) \right\rvert = \frac{r}{2}(p'+1)+c' = p'-p+\frac{r}{2}(p+1)+c \]\[\leq \frac{r}{2}(p+1)+c = \left\lvert E\left(B\left(\frac{r}{2},r,p\right) \vee P_{c}\right) \right\rvert\] where the last inequality follows since $p' \leq p$.
\end{proof}

This gives us everything we need to bound the number of edges for a simple, connected, finite and planar graph $G$ with $r = \text{girth}(G) = \text{circum}(G)$. We first consider the case that $r$ is even.

\begin{lemma}
\label{girthevenprelim}
Let $G$ be a simple, connected, finite and planar graph with $n \geq 4$ vertices and $r = \text{girth}(G) = \text{circum}(G)$ where $r$ is even. Let \[p = \left\lfloor \frac{n-\frac{r}{2}-1}{\frac{r}{2}-1}\right\rfloor \text{ } and \text{ } c = n-2-\left(\frac{r}{2}-1\right)\left(p+1\right)\] Then \[|E(G)| \leq \left\lvert E\left(B\left(\frac{r}{2},r,p\right) \vee P_c\right)\right\rvert.\]
\end{lemma} 

\begin{proof}
First by Lemma ~\ref{vertandedgebook},\[\left\lvert V\left(B\left(\frac{r}{2},r,p\right)\right)\right\rvert = \left(\frac{r}{2}-1\right)(p+1)+2\leq \left(\frac{r}{2}-1\right)\left(\frac{n-\frac{r}{2}-1}{\frac{r}{2}-1}+1\right)+2 = n.\] Therefore by Lemma ~\ref{wedgevertandedge} and Lemma ~\ref{vertandedgebook}, $c \geq 0$, and so $P_c$ is well defined. Moreover, it follows that $\left\lvert V\left(B\left(\frac{r}{2},r,p\right) \vee P_c\right)\right\rvert = n$. Now, we show that the number of vertices in the graph $B\left(\frac{r}{2},r,p+1\right)$ is greater than $n$. In particular by Lemma ~\ref{booktree}, this implies that any other graph $H$ of the form $B \vee T$ where $B$ is a generalised book graph and $T$ is a tree with $n = |V(H)| = \left\lvert V\left(B\left(\frac{r}{2},r,p\right) \vee P_c\right)\right\rvert$ has \begin{equation}\label{bookwedgetreemax}|E(H)| \leq \left\lvert E\left(B\left(\frac{r}{2},r,p\right) \vee P_c\right)\right\rvert.\end{equation} By Lemma ~\ref{vertandedgebook},\[\left\lvert V\left(B\left(\frac{r}{2},r,p+1\right)\right)\right\rvert = \frac{r}{2}+1+\left(\left\lfloor \frac{n-\frac{r}{2}-1}{\frac{r}{2}-1}\right\rfloor+1\right)\left(\frac{r}{2}-1\right) \]\[> \frac{r}{2}+1+\left( \frac{n-\frac{r}{2}-1}{\frac{r}{2}-1}-1+1\right)\left(\frac{r}{2}-1\right) = n.\]

Since we are only considering the number of edges in $G$, by Lemma ~\ref{invariantvertexchoice} and Construction ~\ref{planargraphconstruction}, it suffices to consider the graph $G' = \bigvee_{i=1}^m \mathcal{F}_i \vee P_j$ where $\mathcal{F}_i$ are the strict face components of $G$ and $P_j$ is a path graph. Corollary ~\ref{generalcharacterisation} implies that $\mathcal{F}_i = B(\frac{r}{2},r,p_i)$ for $p_i \geq 1$, with the convention that $B(\frac{r}{2},r,1) = C_r$. 

Now consider $|E(G')| = \left\lvert E\left(\bigvee_{i=1}^m \mathcal{F}_i \vee P_k\right)\right\rvert$. Observe that by definition, the wedge sum is associative and commutative, up to isomorphism. Applying Lemma ~\ref{twobookgraph} to $\mathcal{F}_{m-1} \vee \mathcal{F}_{m}$, we obtain \[\left\lvert E\left(\bigvee\limits_{i=1}^m \mathcal{F}_i \vee P_k\right)\right\rvert \leq \left\lvert E\left(\bigvee\limits_{i=1}^{m-2} \mathcal{F}_i \vee B' \vee P_{c_1} \vee P_k\right)\right\rvert \] for some generalised book graph $B'$ and path graph $P_{c_1}$. Moreover, Lemma ~\ref{invariantvertexchoice} implies that $\left\lvert E\left(\bigvee_{i=1}^{m-2} \mathcal{F}_i \vee B' \vee P_{c_1} \vee P_k\right)\right\rvert = \left\lvert E\left(\bigvee_{i=1}^{m-2} \mathcal{F}_i \vee B' \vee P' \right)\right\rvert$ for some path graph $P'$. Iterate this process by considering the last strict face component in the wedge summand and $B'$ to obtain \[\left\lvert E\left(\bigvee\limits_{i=1}^m \mathcal{F}_i \vee P_k\right)\right\rvert \leq \left\lvert E(\overline{B} \vee \overline{P})\right\rvert\] for some generalised book graph $\overline{B}$ and path graph $\overline{P}$. However, by (\ref{bookwedgetreemax}) \[E(\overline{B} \vee \overline{P}) \leq \left\lvert E\left(B\left(\frac{r}{2},r,p\right) \vee P_c\right)\right\rvert.\]
\end{proof}

The case for $r$ odd is similar.

\begin{lemma}
\label{girthoddprelim}
Let $G$ be a simple, connected, finite and planar graph with $n \geq 3$ vertices and $r = \text{girth}(G) = \text{circum}(G)$ where $r$ is odd. Let \[p' = \left\lfloor \frac{n-1}{r-1}\right\rfloor.\] Then \[|E(G)| \leq \left\lvert E\left(\bigvee\limits_{i=1}^{p'} C_r \vee P_{n-1+p'(1-r)}\right)\right\rvert.\]
\end{lemma} 
\begin{proof}[Sketch]
It follows from Lemma ~\ref{wedgevertandedge}, \[n-1+p'(1-r) \geq 0 \text{ and } \left\lvert V\left(\bigvee_{i=1}^{p'} C_r \vee P_{n-1+p'(1-r)}\right)\right\rvert = n.\] Similar to Lemma ~\ref{booktree}, it can be shown that for graphs $\bigvee_{i=1}^{k} C_r \vee P_{l}$ and $\bigvee_{i=1}^{k'} C_r \vee P_{l'}$ with $k \geq k'$, $l \leq l'$ and $\left\lvert V\left(\bigvee_{i=1}^{k} C_r \vee P_{l}\right)\right\rvert = \left\lvert V\left(\bigvee_{i=1}^{k'} C_r \vee P_{l}\right)\right\rvert$, that \begin{equation}\label{cyclewedgepathstart}\left\lvert E\left(\bigvee_{i=1}^{k'} C_r \vee P_{l'}\right)\right\rvert \leq \left\lvert E\left(\bigvee_{i=1}^{k} C_r \vee P_{l}\right)\right\rvert.\end{equation}

By Lemma ~\ref{invariantvertexchoice} and Construction ~\ref{planargraphconstruction}, it suffices to consider the graph $G' = \bigvee_{i=1}^m \mathcal{F}_i \vee P_j$ where $\mathcal{F}_i$ are the strict face components of $G$ and $P_j$ is a path graph. Corollary ~\ref{generalcharacterisation} implies that $\mathcal{F}_i = C_r$. Therefore by (\ref{cyclewedgepathstart}), to maximise $E(G')$ we wish to maximise $m$ which is achieved by $m = p'$.
\end{proof}

Let $G$ be a simple, connected, finite and planar graph with $n$ vertices and $r = \text{girth}(G) = \text{circum}(G)$. Lemma ~\ref{girthevenprelim} and Lemma ~\ref{girthoddprelim} gives us a sharp upper bound for $|E(G)|$.

\begin{theorem}
\label{upperboundwithr}
Let $G$ be a simple, connected, finite and planar graph with $n$ vertices and $r = \text{girth}(G) = \text{circum}(G)$. Then if $r$ is even and $n \geq 4$, \[|E(G)| \leq n-1+\left\lfloor\frac{n-\frac{r}{2}-1}{\frac{r}{2}-1}\right\rfloor,\] and if $r$ is odd and $n \geq 3$, \[|E(G)| \leq n-1+\left\lfloor\frac{n-1}{r-1}\right\rfloor.\]
\end{theorem}
\begin{proof}
For $r$ even, Lemma ~\ref{girthevenprelim} shows that the graph $G' = B\left(\frac{r}{2},r,p\right) \vee P_c$ where $p = \left\lfloor \frac{n-\frac{r}{2}-1}{\frac{r}{2}-1}\right\rfloor$ and $c = n-2-\left(\frac{r}{2}-1\right)\left(p+1\right)$ has $|E(G)| \leq |E(G')|$ for all graphs $G$ with $n \geq 4$ vertices and $r = \text{girth}(G) = \text{circum}(G)$ for $r$ even. By Lemma ~\ref{wedgevertandedge} and Lemma ~\ref{vertandedgebook}, \[|E(G')| = n-1+\left\lfloor\frac{n-\frac{r}{2}-1}{\frac{r}{2}-1}\right\rfloor.\] 

For $r$ odd, Lemma ~\ref{girthoddprelim} shows that the graph $G' = \bigvee_{i=1}^{p'} C_r \vee P_{n-1+p'(1-r)}$ where $p' = \left\lfloor \frac{n-1}{r-1}\right\rfloor$ has $|E(G)| \leq |E(G')|$ for all graphs $G$ with $n \geq 3$ vertices and $r = \text{girth}(G) = \text{circum}(G)$ for $r$ odd. By Lemma ~\ref{wedgevertandedge} and Lemma ~\ref{vertandedgebook}, \[|E(G')| = n-1+\left\lfloor\frac{n-1}{r-1}\right\rfloor.\] 
\end{proof}

Theorem ~\ref{upperboundwithr} can be used to prove an upper bound that is independent of $r$.

\begin{corollary}
\label{upperboundwithoutr}
Let $G$ be a simple, connected, finite and planar graph with $n\geq 4$ vertices and $r = \text{girth}(G) = \text{circum}(G)$. Then \[|E(G)| \leq 2n-4.\]
\end{corollary}

\begin{proof}
By Theorem ~\ref{upperboundwithr}, if $r$ is even, then $|E(G)| \leq n-1+\left\lfloor\frac{n-\frac{r}{2}-1}{\frac{r}{2}-1}\right\rfloor$. This is monotonically decreasing for $r \geq 4$ and so \[|E(G)| \leq n-1+\left\lfloor\frac{n-\frac{4}{2}-1}{\frac{4}{2}-1}\right\rfloor = 2n-4.\] If $r$ is odd, then by Theorem ~\ref{upperboundwithr}, $|E(G)| \leq n-1+\left\lfloor\frac{n-1}{r-1}\right\rfloor$. This is monotonically decreasing for $r \geq 3$ and so \[|E(G)| \leq n-1+\left\lfloor\frac{n-1}{2}\right\rfloor.\] 

Consider $n-1+\left\lfloor\frac{n-1}{2}\right\rfloor$. By definition of floor, $n-1+\left\lfloor\frac{n-1}{2}\right\rfloor \leq n-1+\frac{n-1}{2} = \frac{3}{2}n-\frac{3}{2}$. For all $n \geq 5$, $\frac{3}{2}n-\frac{3}{2} \leq 2n-4$ and so the bound holds for all $r \geq 3$ and $n \geq 5$. For $n=4$ and $r$ odd, by Corollary ~\ref{generalcharacterisation}, the only possibility for $G$ is $G = C_3 \vee P_2$ which has $4 \leq 2\cdot 4-4 = 4$ edges. Therefore, the bound holds for all $n \geq 4$ and $r \geq 3$.
\end{proof}

Theorem ~\ref{upperboundwithr} and Corollary ~\ref{upperboundwithoutr} can be used to prove the existence of two cycles of different lengths in a planar graph.

\begin{example}
Let $G$ be a simple, connected, finite and planar graph with 16 vertices and 29 edges. Suppose all the cycles in $G$ are of the same length, in other words, $\text{girth}(G) = \text{circum}(G)$. Then by Corollary ~\ref{upperboundwithoutr}, $|E(G)| \leq 2\cdot 16 -4 = 28$ which is a contradiction. Therefore, $G$ must contain two cycles of different lengths.
\end{example}

\begin{example}
Let $G$ be a simple, connected, finite and planar graph with 16 vertices, 22 edges and a cycle of length 6. Suppose all the cycles in $G$ are of the same length, in other words, $6=\text{girth}(G) = \text{circum}(G)$. Then by Corollary ~\ref{upperboundwithr}, $|E(G)| \leq 16-1+\left\lfloor \frac{16-3-1}{3-1}\right\rfloor = 21$ which is a contradiction. Therefore, $G$ must contain two cycles of different lengths.
\end{example}

\end{document}